\documentclass[12pt]{article}
\usepackage{amsmath,amsthm,amssymb}
\usepackage{amssymb,latexsym}

\newtheorem{theorem}{Theorem}
\newtheorem{conjecture}{Conjecture}
\newtheorem{lemma}{Lemma}

\begin{document}

\baselineskip=17pt

\title{\bf On an tangent equation by primes}

\author{\bf S. I. Dimitrov}

\date{2021}

\maketitle

\begin{abstract}
In this paper we introduce a new diophantine equation with prime numbers.
Let $[\, \cdot\,]$ be the floor function.
We prove that when $1<c<\frac{23}{21}$ and $\theta>1$ is a fixed, then
every sufficiently large positive integer $N$ can be represented in the form
\begin{equation*}
N=\big[p^c_1\tan^\theta(\log p_1)\big]+ \big[p^c_2\tan^\theta(\log p_2)\big]+ \big[p^c_3\tan^\theta(\log p_3)\big]\,,
\end{equation*}
where $p_1,\, p_2,\, p_3$ are prime numbers.
We also establish an asymptotic formula for the number of such representations. 
\quad\\
{\bf  2020 Math.\ Subject Classification}:  11P55 $\cdot$ 11L07
\end{abstract}

\section{Introduction and main result}
\indent

Analytical number theorists remember 1937 well when Vinogradov \cite{Vinogradov1}
proved the ternary Goldbach problem.
He showed that every sufficiently large odd integer $N$ can be represented  in the form
\begin{equation*}
N=p_1+p_2+p_3,
\end{equation*}
where $p_1,\, p_2,\, p_3$ are prime numbers.

Source of detailed proof of Vinogradov's theorem, beginning with an historical
perspective along with an overview of essential lemmas and theorems,
can be found in monograph of Rassias \cite{Rassias}.

In 1995 Laporta and Tolev  \cite{Laporta-Tolev} investigated an analogue of the Goldbach-Vinogradov theorem.
They considered the diophantine equation
\begin{equation*}
N=[p^c_1]+[p^c_2]+[p^c_3]\,,
\end{equation*}
where $p_1,\, p_2,\, p_3$ are primes.
For $1<c <\frac{17}{16}$ they showed that for the sum
\begin{equation*}
R(N)=\sum\limits_{N=[p^c_1]+[p^c_2]+[p^c_3]}\log p_1\log p_2\log p_3
\end{equation*}
the asymptotic formula
\begin{equation}\label{RNasymptoticformula}
R(N)=\frac{\Gamma^3(1 + 1/c)}{\Gamma(3/c)}N^{3/c-1}
+\mathcal{O}\Big(N^{3/c-1}\exp\big(-(\log N)^{1/3-\varepsilon}\big)\Big)
\end{equation}
holds.

Afterwards the result of Laporta and Tolev was improved by
Kumchev and Nedeva \cite{Kumchev-Nedeva} to $1<c <\frac{12}{11}$,
by Zhai and Cao  \cite{Zhai-Cao} to $1<c<\frac{258}{235}$,
by Cai \cite{Cai} to $1<c <\frac{137}{119}$,
by Zhang and Li \cite{Zhang-Li3} to $1<c <\frac{3113}{2703}$,
by Baker \cite{Baker} to $1<c<\frac{3581}{3106}$ and this is the best result up to now.

On the other hand recently the author \cite{Dimitrov} showed that when $1<c <\frac{10}{9}$
and $N$ is a sufficiently large  positive number, then for any fixed $\theta>1$, the tangent inequality
\begin{equation*}
\big|p^c_1\tan^\theta(\log p_1)+  p^c_2\tan^\theta(\log p_2)+ p^c_3\tan^\theta(\log p_3) -N\big|
<\left(\frac{2^\theta N}{3^{\theta+3}}\right)^{\frac{1}{c^2}\left(c-\frac{10}{9}\right)}
\end{equation*}
has a solution in prime numbers $p_1,\,p_2,\,p_3$.

Motivated by these results in this paper we introduce new diophantine equation with prime numbers.
Let $N$ is a sufficiently large positive integer and $X$ is an arbitrary solution of the equation
\begin{equation}\label{XN}
\pi\left[\frac{\log X}{\pi}\right]+\arctan2=\frac{1}{c}\log\frac{ N}{2^\theta} \,.
\end{equation}
Define
\begin{align}\label{Delta1}
&\Delta_1=e^{\pi\big[\frac{\log X}{\pi}\big]+\arctan1}\,;\\
\label{Delta2}
&\Delta_2=e^{\pi\big[\frac{\log X}{\pi}\big]+\arctan2}\,;\\       \label{Gamma}
&\Gamma= \sum\limits_{\Delta_1<p_1,p_2,p_3\leq\Delta_2
\atop{N=[p^c_1\tan^\theta(\log p_1)]+[p^c_2\tan^\theta(\log p_2)]+ [p^c_3\tan^\theta(\log p_3)]}}\log p_1\log p_2\log p_3\,.
\end{align}

\begin{theorem} Let $N$ is a sufficiently large positive integer, $\theta>1$ is a fixed
and $X$ is an arbitrary solution of the equation  \eqref{XN}. Then for any fixed $1<c<\frac{23}{21}$, the asymptotic formula
\begin{equation}\label{Asymptoticformula}
\Gamma=\frac{\Delta_2^{1-c}}{2^\theta c+5\theta2^{\theta-1}}X^2
+\mathcal{O}\Big(X^{3-c}\exp\big(-(\log X)^{1/3-\varepsilon}\big)\Big)
\end{equation}
holds.   Here $\Delta_2$ is defined by \eqref{Delta2}.
\end{theorem}
In addition we have the following conjecture.
\begin{conjecture}Let $N$ is a sufficiently large positive integer and $\theta>1$ is a fixed.
There exists $c_0>1$ such that for any fixed $1<c<c_0$, the tangent equation
\begin{equation*}
N=\big[p^c_1\tan^\theta(\log p_1)\big]+ \big[p^c_2\tan^\theta(\log p_2)\big]
\end{equation*}
has a solution in prime numbers $p_1,\,p_2$.
\end{conjecture}

\section{Notations}
\indent

Assume that $N$ is a sufficiently large  positive integer.
By $\varepsilon$ we denote an arbitrary small positive constant, not the same in all appearances.
The letter $p$  with or without subscript will always denote prime number.
We denote by $\Lambda(n)$ von Mangoldt's function.
Moreover $e(y)=e^{2\pi i y}$.
As usual $[t]$ and $\{t\}$ denote the integer part, respectively, the fractional part of $t$.
We recall that $t=[t]+\{t\}$ and $\|t\|=\min(\{t\}_,1-\{t\})$.
We denote by $\tau _k(n)$ the number of solutions of the equation $m_1m_2\ldots m_k$ $=n$ in natural numbers $m_1,\,\ldots,m_k$.
Throughout this paper we suppose that $1<c<\frac{23}{21}$.
Assume that $\theta>1$ is a fixed. Consider the function $t(y)$ defined by
\begin{equation}\label{ylogyt}
t=y^c\tan^\theta(\log y)
\end{equation}
for
\begin{equation}\label{yDelta1Delta2}
y\in[\Delta_1, \Delta_2]\,.
\end{equation} The first derivative of $y$ as implicit function of $t$ is
\begin{equation}\label{y'}
y'=\frac{y^{1-c}}{\big(c\tan(\log y)+\theta\sec^2(\log y)\big)\tan^{\theta-1}(\log y)}\,.
\end{equation}
Denote
\begin{align}
\label{tau}
&\tau=X^{1-c-\varepsilon}\,;\\
\label{N1}
&N_1=\Delta_1^c\tan^\theta(\log \Delta_1)   \,;\\
\label{Salpha}
&S(\alpha)=\sum\limits_{\Delta_1<p\leq \Delta_2} e\big(\alpha [p^c\tan^\theta(\log p)]\big)\log p\,;\\
\label{Theta}
&\Theta(\alpha)=\sum\limits_{N_1<m\leq N}\frac{y^{1-c}(m)}{\Big(c\tan\big(\log y(m)\big)+\theta\sec^2\big(\log y(m)\big)\Big)
\tan^{\theta-1}\big(\log y(m)\big)}\, e(m\alpha)\,;
\end{align}
\begin{align}
\label{Gamma1}
&\Gamma_1=\int\limits_{-\tau}^{\tau}S^3(\alpha)e(-N\alpha)\,d\alpha\,;\\
\label{Gamma2}
&\Gamma_2=\int\limits_{\tau}^{1-\tau}S^3(\alpha)e(-N\alpha)\,d\alpha\,;\\
\label{Psik}
&\Psi_k=\int\limits_{-1/2}^{1/2}\Theta^k(\alpha)e(-N\alpha)\,d\alpha, \quad k=1,\, 2,\, 3, \ldots\,; \\ \label{Psiwidetilde}
&\widetilde{\Psi}=\int\limits_{-\tau}^{\tau}\Theta^3(\alpha)e(-N\alpha)\,d\alpha\,.
\end{align}

\section{Lemmas}
\indent

\begin{lemma}\label{ThetaS}
Let $f(x)$ be a real differentiable function in the interval $[a,b]$.
If $f'(x)$ is a monotonous and satisfies $|f'(x)|\leq\theta<1$.
Then we have
\begin{equation*}
\sum_{a<n\le b}e(f(n))=\int\limits_{a}^{b}e(f(x))\,dx+\mathcal{O}(1)\,.
\end{equation*}
\end{lemma}
\begin{proof}
See (\cite{Titchmarsh},  Lemma 4.8).
\end{proof}

\begin{lemma}\label{Squareoutlemma}
For any complex numbers $a(n)$ we have
\begin{equation*}
\bigg|\sum_{a<n\le b}a(n)\bigg|^2
\leq\bigg(1+\frac{b-a}{Q}\bigg)\sum_{|q|< Q}\bigg(1-\frac{|q|}{Q}\bigg)
\sum_{a<n,\, n+q\leq b}a(n+q)\overline{a(n)},
\end{equation*}
where $Q$ is any positive integer.
\end{lemma}
\begin{proof}
See (\cite{Iwaniec-Kowalski}, Lemma 8.17).
\end{proof}

\begin{lemma}\label{GrahamandKolesnik}
Let $k \geq0$ be an integer.
Suppose that $f(t)$ has $k+2$ continuous derivatives on $I$, and that $I \subseteq(N,2N]$.
Assume also that there is some constant $F$ such that
\begin{equation}\label{frFNR}
|f^{(r)}(t)|\asymp F N^{-r}
\end{equation}
for $r = 1, \ldots, k + 2$. Let $Q = 2^k$. Then
\begin{equation*}
\bigg|\sum_{n\in I}e(f(n))\bigg|\ll F^{\frac{1}{4Q-2}} N^{1-\frac{k+2}{4Q-2}}  +F^{-1} N\,.
\end{equation*}
The implied constant depends only upon the implied constants in \eqref{frFNR}.
\end{lemma}
\begin{proof}
See (\cite{Graham-Kolesnik}, Theorem 2.9).
\end{proof}

\begin{lemma}\label{Buriev} Let $x,y\in\mathbb{R}$ and $H\geq3$.
Then the formula
\begin{equation*}
e(-x\{y\})=\sum\limits_{|h|\leq H}c_h(x)e(hy)+\mathcal{O}\left(\min\left(1, \frac{1}{H\|y\|}\right)\right)
\end{equation*}
holds. Here
\begin{equation*}
c_h(x)=\frac{1-e(-x)}{2\pi i(h+x)}\,.
\end{equation*}
\end{lemma}
\begin{proof}
See (\cite{Buriev}, Lemma 12).
\end{proof}

\begin{lemma}\label{Heath-Brown} Let $G(n)$ be a complex valued function.
Assume further that
\begin{align*}
&P>2\,,\quad P_1\le 2P\,,\quad  2\le U<V\le Z\le P\,,\\
&U^2\le Z\,,\quad 128UZ^2\le P_1\,,\quad 2^{18}P_1\le V^3\,.
\end{align*}
Then the sum
\begin{equation*}
\sum\limits_{P<n\le P_1}\Lambda(n)G(n)
\end{equation*}
can be decomposed into $O\Big(\log^6P\Big)$ sums, each of which is either of Type I
\begin{equation*}
\mathop{\sum\limits_{M<m\le M_1}a(m)\sum\limits_{L<l\le L_1}}_{P<ml\le P_1}G(ml)
\end{equation*}
and
\begin{equation*}
\mathop{\sum\limits_{M<m\le M_1}a(m)\sum\limits_{L<l\le L_1}}_{P<ml\le P_1}G(ml)\log l\,,
\end{equation*}
where
\begin{equation*}
L\ge Z\,,\quad M_1\le 2M\,,\quad L_1\le 2L\,,\quad a(m)\ll \tau _5(m)\log P
\end{equation*}
or of Type II
\begin{equation*}
\mathop{\sum\limits_{M<m\le M_1}a(m)\sum\limits_{L<l\le L_1}}_{P<ml\le P_1}b(l)G(ml)
\end{equation*}
where
\begin{equation*}
U\le L\le V\,,\quad M_1\le 2M\,,\quad L_1\le 2L\,,\quad
a(m)\ll \tau _5(m)\log P\,,\quad b(l)\ll \tau _5(l)\log P\,.
\end{equation*}
\end{lemma}
\begin{proof}
See (\cite{Heath1}).
\end{proof}

\begin{lemma}\label{Expansion}
For any real number $t$ and $H\geq1$, there holds
\begin{equation*}
\min\left(1, \frac{1}{H\|t\|}\right)
=\sum\limits_{h=-\infty}^{+\infty}a_he(h t)\,,
\end{equation*}
where
\begin{equation*}
a_h\ll\min\left(\frac{\log 2H}{H}, \frac{1}{|h|}, \frac{H}{|h|^2}, \right).
\end{equation*}
\begin{proof}
See (\cite{Heath2}, p. 245).
\end{proof}
\end{lemma}

\begin{lemma}\label{intStautau}
For the sum denoted by \eqref{Salpha} we have
\begin{align*}
&\emph{(i)}\quad\quad\quad\;\,
\int\limits_{-\tau}^\tau|S(\alpha)|^2\,dt\,\ll X^{2-c}\log^2X\,,\\ &\emph{(ii)}\quad\quad\quad\,
\int\limits_{0}^{1}|S(\alpha)|^2\,dt\ll X\log X\,.
\end{align*}
\end{lemma}
\begin{proof}
It follows from the arguments used in (\cite{Dimitrov}, Lemma 8).
\end{proof}

\section{Proof of the Theorem}
\indent

From \eqref{Gamma}, \eqref{Salpha},  \eqref{Gamma1}  and \eqref{Gamma2} we have
\begin{equation}\label{GammaGamma12}
\Gamma=\int\limits_{0}^{1}S^3(\alpha)e(-N\alpha)\,d\alpha=\Gamma_1+\Gamma_2\,.
\end{equation}

\subsection{Estimation of $\mathbf{\Gamma_1}$}

We write
\begin{equation}\label{Gamma1decomp}
\Gamma_1=(\Gamma_1-\widetilde{\Psi})+(\widetilde{\Psi}-\Psi_3)+\Psi_3\,.
\end{equation}
Bearing in mind \eqref{Theta} and \eqref{Psik} we obtain
\begin{align*}
\Psi_1&=\int\limits_{-1/2}^{1/2}\Theta(\alpha)e(-N\alpha)\,d\alpha\\
&=\frac{y^{1-c}(N)}{\Big(c\tan\big(\log y(N)\big)+\theta\sec^2\big(\log y(N)\big)\Big)\tan^{\theta-1}\big(\log y(N)\big)}\,.
\end{align*}
Suppose that
\begin{equation}\label{Supposition}
\Psi_k=\frac{y^{1-c}(N)}{\Big(c\tan\big(\log y(N)\big)+\theta\sec^2\big(\log y(N)\big)\Big)\tan^{\theta-1}\big(\log y(N)\big)}\,
X^{k-1}+\mathcal{O}\Big(X^{k-2}\Big)
\end{equation}
for
\begin{equation*}
k\geq2\,.
\end{equation*}
From \eqref{XN} and \eqref{ylogyt} it follows that
\begin{equation}\label{yNX}
y(N)=\Delta_2\,.
\end{equation}                                                       
Now  \eqref{Delta1}, \eqref{Delta2}, \eqref{ylogyt}, \eqref{yDelta1Delta2}, \eqref{N1},  \eqref{Supposition} and \eqref{yNX} yields
\begin{align*}
&\Psi_{k+1}= \\
&\sum\limits_{N_1<m\leq N}\frac{y^{1-c}(m)}{\Big(c\tan\big(\log y(m)\big)
+\theta\sec^2\big(\log y(m)\big)\Big)\tan^{\theta-1}\big(\log y(m)\big)}  \\
&\times\left(\mathop{\sum\limits_{N_1<m_1\leq N-m}\cdots\sum\limits_{N_1<m_k\leq N-m}}_
{m_1+\cdots+m_k=N-m}\frac{y^{1-c}(m_1)}{\Big(c\tan\big(\log y(m_1)\big)+\theta\sec^2\big(\log y(m_1)\big)\Big)\tan^{\theta-1}\big(\log y(m_1)\big)}\right.\\
&\left.\cdots\frac{y^{1-c}(m_2)}{\Big(c\tan\big(\log y(m_2)\big)+\theta\sec^2\big(\log y(m_2)\big)\Big)\tan^{\theta-1}\big(\log y(m_2)\big)}\right)\\
&=\sum\limits_{N_1<m\leq N}\frac{y^{1-c}(m)}{\Big(c\tan\big(\log y(m)\big)+\theta\sec^2\big(\log y(m)\big)\Big)\tan^{\theta-1}\big(\log y(m)\big)}\\
&\times\left(\frac{y^{1-c}(N-m)}{\Big(c\tan\big(\log y(N-m)\big)+\theta\sec^2\big(\log y(N-m)\big)\Big)\tan^{\theta-1}\big(\log y(N-m)\big)}
X^{k-1}+\mathcal{O}\big(X^{k-2}\big)\right)\\
&=\sum\limits_{N_1<m< N-N_1}\frac{y^{1-c}(m)}{\Big(c\tan\big(\log y(m)\big)+\theta\sec^2\big(\log y(m)\big)\Big)\tan^{\theta-1}\big(\log y(m)\big)}\\
&\frac{y^{1-c}(N-m)}{\Big(c\tan\big(\log y(N-m)\big)+\theta\sec^2\big(\log y(N-m)\big)\Big)\tan^{\theta-1}\big(\log y(N-m)\big)}
X^{k-1}+\mathcal{O}\big(X^{k-1}\big)\\
&= \frac{y^{1-c}(N)}{\Big(c\tan\big(\log y(N)\big)+\theta\sec^2\big(\log y(N)\big)\Big)\tan^{\theta-1}\big(\log y(N)\big)}\,
X^k+\mathcal{O}\big(X^{k-1}\big)\,.
\end{align*}
Consequently the supposition \eqref{Supposition} is true. 
Bearing in mind \eqref{Delta2}, \eqref{Supposition} and \eqref{yNX} we deduce
\begin{equation}\label{PsikAsymptoticformula}
\Psi_k= \frac{\Delta_2^{1-c}}{2^\theta c+5\theta2^{\theta-1}}\,
X^{k-1}+\mathcal{O}\big(X^{k-2}\big) \quad \mbox{for} \quad k\geq2.
\end{equation}
Now the asymptotic formula \eqref{PsikAsymptoticformula} gives us
\begin{equation}\label{Psi3Asymptoticformula}
\Psi_3=\frac{\Delta_2^{1-c}}{2^\theta c+5\theta2^{\theta-1}}X^2+\mathcal{O}\big(X\big)\,.
\end{equation}
From \eqref{Gamma1} and \eqref{Psiwidetilde} we get
\begin{align}\label{Gamma1Psi1}
|\Gamma_1-\widetilde{\Psi}|&\ll\int\limits_{-\tau}^{\tau}
\big|S^3(\alpha)-\Theta^3(\alpha)\big|\,d\alpha\nonumber\\
&\ll\max\limits_{|\alpha|\leq \tau}\big|S(\alpha)-\Theta(\alpha)\big|
\left(\int\limits_{-\tau}^{\tau}|S(\alpha)|^2\,d\alpha
+\int\limits_{-1/2}^{1/2}|\Theta(\alpha)|^2\,d\alpha \right)\,.
\end{align}
By \eqref{Delta1}, \eqref{Delta2}, \eqref{ylogyt}, \eqref{yDelta1Delta2},
\eqref{N1}, \eqref{Theta} and \eqref{yNX} we obtain
\begin{equation}\label{1IntTheta}
\int\limits_{-1/2}^{1/2}|\Theta(\alpha)|^2\,d\alpha\ll  X^{2-c}\,.
\end{equation}
Now we shall estimate from above $|S(\alpha)-\Theta(\alpha)|$  for $|\alpha|\leq \tau$.\\
Using \eqref{Delta1}, \eqref{Delta2}, \eqref{tau} and \eqref{Salpha} we write
\begin{align}\label{Salphaest1}
S(\alpha)&=\sum\limits_{\Delta_1<p\leq \Delta_2} e\big(\alpha p^c\tan^\theta(\log p)\big)\log p
+\mathcal{O}\big(\tau X\big) \nonumber\\
&=\sum\limits_{\Delta_1<n\leq \Delta_2} \Lambda(n)e\big(\alpha n^c\tan^\theta(\log n)\big)
+\mathcal{O}\big( X^{1/2}\big)+\mathcal{O}\big(\tau X\big)\nonumber\\
&=\sum\limits_{\Delta_1<n\leq \Delta_2} \Lambda(n)e\big(\alpha n^c\tan^\theta(\log n)\big)
+\mathcal{O}\big(X^{1-\varepsilon}\big).
\end{align}
From $|\alpha|\leq \tau$, $y\geq N_1$  and Lemma \ref{ThetaS} we have that
\begin{equation}\label{SumInt}
\sum_{N_1<m\le y}e(m\alpha)=\int\limits_{N_1}^{y}e(\alpha t)\,dt+\mathcal{O}(1).
\end{equation}
Using \eqref{Delta1}, \eqref{Delta2}, \eqref{ylogyt}, \eqref{yDelta1Delta2}, \eqref{y'},  \eqref{tau},
\eqref{N1},  \eqref{Theta}, \eqref{yNX}, \eqref{SumInt} and partial summation we find
\begin{align}\label{LambdaTheta}
&\sum\limits_{\Delta_1<n\leq \Delta_2} \Lambda(n)e\big(\alpha n^c\tan^\theta(\log n)\big)
=\int\limits_{\Delta_1}^{\Delta_2}e\big(\alpha y^c\tan^\theta(\log y)\big)\,d\bigg(\sum\limits_{\Delta_1<n\leq y} \Lambda(n)\bigg) \nonumber \\
&=\int\limits_{\Delta_1}^{\Delta_2}e\big(\alpha y^c\tan^\theta(\log y)\big)\,dy+\mathcal{O}\bigg(X\exp\Big(-(\log X)^{1/3}\Big)\bigg) \nonumber \\
&=\int\limits_{N_1}^Ne(\alpha t)
\frac{y^{1-c}(t)}{\Big(c\tan\big(\log y(t)\big)+\theta\sec^2\big(\log y(t)\big)\Big)\tan^{\theta-1}\big(\log y(t)\big)}\,dt  \nonumber \\
&+\mathcal{O}\bigg(X\exp\Big(-(\log X)^{1/3}\Big)\bigg) \nonumber \\
&=\int\limits_{N_1}^N\frac{y^{1-c}(t)}{\Big(c\tan\big(\log y(t)\big)+\theta\sec^2\big(\log y(t)\big)\Big)\tan^{\theta-1}\big(\log y(t)\big)}\,
d\left(\int\limits_{N_1}^te(\alpha u)\,du\right) \nonumber \\
&+\mathcal{O}\bigg(X\exp\Big(-(\log X)^{1/3}\Big)\bigg)\nonumber \\
&=\int\limits_{N_1}^N\frac{y^{1-c}(t)}{\Big(c\tan\big(\log y(t)\big)+\theta\sec^2\big(\log y(t)\big)\Big)\tan^{\theta-1}\big(\log y(t)\big)}\, d\Bigg(\sum_{N_1<m\le t}e( m\alpha)+\mathcal{O}(1)\Bigg)\nonumber \\
&+\mathcal{O}\bigg(X\exp\Big(-(\log X)^{1/3}\Big)\bigg)\nonumber \\
&=\sum\limits_{N_1< m\leq N}\frac{y^{1-c}(m)}{\Big(c\tan\big(\log y(m)\big)+\theta\sec^2\big(\log y(m)\big)\Big)\tan^{\theta-1}\big(\log y(m)\big)}
\,e(m\alpha)\nonumber \\
&+\mathcal{O}\bigg(X\exp\Big(-(\log X)^{1/3}\Big)\bigg)\nonumber \\
&=\Theta(\alpha)+\mathcal{O}\bigg(X\exp\Big(-(\log X)^{1/3}\Big)\bigg)\,.
\end{align}
By \eqref{Salphaest1} and \eqref{LambdaTheta} it follows that
\begin{equation}\label{SalphaThetaalpha}
\max\limits_{|\alpha|\leq \tau}|S(\alpha)-\Theta(\alpha)|\ll X\exp\Big(-(\log X)^{1/3}\Big)\,.
\end{equation}
Taking into account \eqref{Gamma1Psi1}, \eqref{1IntTheta}, \eqref{SalphaThetaalpha} and Lemma \ref{intStautau} we get
\begin{equation}\label{Gamma1widetildePsiest}
\Gamma_1-\widetilde{\Psi}\ll X^{3-c}\exp\Big(-(\log X)^{1/3-\varepsilon}\Big)\,.
\end{equation}
Using  \eqref{Delta1}, \eqref{Delta2}, \eqref{ylogyt}, \eqref{yDelta1Delta2}, \eqref{y'},  \eqref{tau}, \eqref{N1},
\eqref{Theta}, \eqref{Psik}, \eqref{Psiwidetilde}, \eqref{yNX} and working as in (\cite{Vaughan}, Lemma 2.8)  we deduce
\begin{equation}\label{Psi3-widetildePsi}
\big|\Psi_3-\widetilde{\Psi}\big|\ll\int\limits_{\tau\leq|\alpha|\leq1/2}
|\Theta(\alpha)|^3\,d\alpha\ll\int\limits_{\tau}^{1/2}\alpha^{-\frac{3}{c}}\,d\alpha
\ll X^{3-c-\varepsilon}\,.
\end{equation}
Summarizing \eqref{Gamma1decomp}, \eqref{Psi3Asymptoticformula}, \eqref{Gamma1widetildePsiest}
and \eqref{Psi3-widetildePsi} we obtain
\begin{equation}\label{Gamma1asymptotic}
\Gamma_1=\frac{\Delta_2^{1-c}}{2^\theta c+5\theta2^{\theta-1}}X^2
+\mathcal{O}\bigg(X^{3-c}\exp\Big(-(\log X)^{1/3-\varepsilon}\Big)\bigg)\,.
\end{equation}

\subsection{Estimation of $\mathbf{\Gamma_2}$}

\begin{lemma}\label{SIest} Assume that
\begin{equation}\label{taualphatau1}
\tau \leq \alpha \leq 1-\tau\,.
\end{equation}
Set \begin{equation}\label{SI}
S_I=\sum\limits_{|h|\leq H}c_h(\alpha)\mathop{\sum\limits_{M<m\le M_1}a(m)\sum\limits_{L<l\le L_1}}_{\Delta_1<ml\le \Delta_2}
e\big((h+\alpha) m^cl^c \tan^\theta\big(\log (ml) \big)\big)
\end{equation}
and
\begin{equation}\label{SI'}
S'_I=\sum\limits_{|h|\leq H}c_h(\alpha)\mathop{\sum\limits_{M<m\le M_1}a(m)\sum\limits_{L<l\le L_1}}_{\Delta_1<ml\le \Delta_2}
e\big((h+\alpha) m^cl^c \tan^\theta\big(\log (ml) \big)\big)\log l\,,
\end{equation}
where
\begin{equation}\label{Conditions1}
L\ge X^{\frac{22}{45}}\,,\quad M_1\le 2M\,,\quad L_1\le 2L\,,\quad a(m)\ll \tau _5(m)\log \Delta_1\,,\quad H= X^{\frac{4-3c}{15}} \end{equation}
and $c_h(\alpha)$ denote complex numbers such that $|c_h(\alpha)|\ll (1+|h|)^{-1}$.\\
Then
\begin{equation*}
S_I,\, S'_I\ll  X^{\frac{11+3c}{15}+\varepsilon}\,.
\end{equation*}                             \end{lemma}

\begin{proof}
First we notice that \eqref{Delta1}, \eqref{Delta2}, \eqref{SI}  and \eqref{Conditions1} imply \begin{equation}\label{LMasympX}
LM\asymp X\,.
\end{equation}
Denote
\begin{equation}\label{flm}
f_h(l, m)=m^cl^c \tan^\theta\big(\log (ml) \big)\,.
\end{equation}
By  \eqref{Delta1}, \eqref{SI}, \eqref{Conditions1} and \eqref{flm}  we write
\begin{equation}\label{SIest1}
S_I\ll X^\varepsilon\max\limits_{|\eta|\in [\tau, H+1]}
\sum\limits_{M<m\leq M_1}\bigg|\sum\limits_{L'<l\leq L'_1}e\big(\eta f(l, m)\big)\bigg|\,,
\end{equation}
where
\begin{equation}\label{L'L1'}
L'=\max{\bigg\{L,\frac{\Delta_1}{m}\bigg\}}\,,\quad L_1'=\min{\bigg\{L_1,\frac{\Delta_2}{m}\bigg\}}\,.
\end{equation}
From  \eqref{Conditions1} and \eqref{L'L1'} for the sum in \eqref{SIest1} it follows
\begin{equation}\label{L'andL1'inL2L}
\left|\begin{array}{cccc}
M<m\leq M_1\quad \; \\
L'<l\leq L'_1  \quad \quad \\
\Delta_1<ml\le \Delta_2 \;\;\\
(L', L'_1]\subseteq (L, 2L]
\end{array}\right..
\end{equation}
On the other hand for the function defined by  \eqref{flm} we find
\begin{equation}\label{firstderivativel}
\frac{\partial f(l, m)}{\partial l}=m^cl^{c-1}\tan^{{\theta}-1}\big(\log(ml)\big)
\Big(c\tan\big(\log(ml)\big))+{\theta}\sec^2\big(\log(ml)\big)\Big)
\end{equation}
and
\begin{align}
\label{secondderivativel}
\frac{\partial^2f(l, m)}{\partial l^2} &= m^cl^{c-2}\tan^{\theta-2}\big(\log(ml)\big)
\Big(\big(2\theta\sec^2\big(\log(ml)\big)+c^2-c\big)\tan^2\big(\log(ml)\big)\nonumber\\
&+(2c-1)\theta\sec^2\big(\log(ml)\big)\tan\big(\log(ml)\big)+(\theta^2-\theta)\sec^4\big(\log(ml)\big)\Big)\,.
\end{align}
Now \eqref{Delta1}, \eqref{Delta2}, \eqref{Conditions1}, \eqref{L'andL1'inL2L}, \eqref{firstderivativel} and \eqref{secondderivativel} yields
\begin{equation}\label{firstderivativelasymp}
\frac{\partial f(d,l)}{\partial l}\asymp M^cL^{c-1}
\end{equation}
and
\begin{equation}\label{secondderivativelasymp}
\frac{\partial^2f(d,l)}{\partial l^2}\asymp  M^cL^{c-2}\,.
\end{equation}
Proceeding in the same way we get
\begin{equation}\label{thirdderivativelasymp}
\frac{\partial^3f(d,l)}{\partial l^3}\asymp M^cL^{c-3}\,.
\end{equation}
Using \eqref{tau}, \eqref{taualphatau1}, \eqref{Conditions1},
\eqref{LMasympX}, \eqref{SIest1}, \eqref{L'andL1'inL2L}, \eqref{firstderivativelasymp}, \eqref{secondderivativelasymp},
\eqref{thirdderivativelasymp} and Lemma \ref{GrahamandKolesnik} with  $k=1$ we obtain
\begin{align*}\label{SIest1}
S_I&\ll X^\varepsilon\max\limits_{|\eta|\in [\tau, H+1]}
\sum\limits_{M<m\leq M_1}\bigg(|\eta|^{\frac{1}{6}}M^{\frac{c}{6}} L^{\frac{c}{6}}L^{\frac{1}{2}}
+|\eta|^{-1}M^{-c}L^{1-c}\bigg)\\
&\ll X^\varepsilon\Big(M^{\frac{1}{2}}H^{\frac{1}{6}}X^{\frac{c}{6}+\frac{1}{2}}+\tau^{-1}X^{1-c}\Big)\\
&\ll   X^{\frac{11+3c}{15}+\varepsilon}\,.
\end{align*}
To estimate the sum defined by \eqref{SI'} we apply Abel's summation formula and proceed in the same way to deduce
\begin{equation*}
S_I'\ll   X^{\frac{11+3c}{15}+\varepsilon}\,.
\end{equation*} This proves the lemma.   \end{proof}

\begin{lemma}\label{SIIest} Assume that
\begin{equation}\label{taualphatau2}
\tau \leq \alpha\leq 1-\tau \,.
\end{equation}
Set
\begin{equation}\label{SII}
S_{II}=\sum\limits_{|h|\leq H}c_h(\alpha)\mathop{\sum\limits_{M<m\le M_1}a(m)\sum\limits_{L<l\le L_1}}_{\Delta_1<ml\le \Delta_2}b(l)
e\big((h+\alpha) m^cl^c \tan^\theta\big(\log (ml) \big)\big)\,,
\end{equation}
where
\begin{equation}
\begin{split}\label{Conditions2}
&2^{-11}X^{\frac{1}{45}}\leq L\leq 2^7X^{\frac{1}{3}}\,,\quad M_1\le 2M\,,\quad L_1\le 2L\,,\\
&a(m)\ll \tau _5(m)\log \Delta_1\,,\quad b(l)\ll \tau _5(l)\log \Delta_1 \,,\quad H= X^{\frac{4-3c}{15}}
\end{split}
\end{equation}
and $c_h(\alpha)$ denote complex numbers such that $|c_h(\alpha)|\ll (1+|h|)^{-1}$.\\
Then
\begin{equation*}
S_{II}\ll  X^{\frac{11+3c}{15}+\varepsilon}\,.
\end{equation*}
\end{lemma}

\begin{proof}
First we notice that \eqref{Delta1}, \eqref{Delta2}, \eqref{SII}  and \eqref{Conditions2} give us
\begin{equation}\label{LMasympX2}
LM\asymp X\,.
\end{equation}
From \eqref{Delta1}, \eqref{flm}, \eqref{SII}, \eqref{Conditions2}, \eqref{LMasympX2}, Cauchy's inequality
and Lemma \ref{Squareoutlemma} with $Q= X^{\frac{8-6c}{15}}$ it follows
\begin{align*}
S_{II}&\ll\sum\limits_{|h|\leq H}\big|c_h(\alpha)\big|
\Bigg|\mathop{\sum\limits_{M<m\le M_1}a(m)\sum\limits_{L<l\le L_1}}_{\Delta_1<ml\le \Delta_2}b(l)e\big((h+\alpha)
m^cl^c \tan^\theta\big(\log (ml) \big)\big)\Bigg|  \nonumber\\  &\ll\sum\limits_{|h|\leq H}\big|c_h(\alpha)\big|\left(\sum\limits_{M<m\le M_1}|a(m)|^2\right)^{\frac{1}{2}}
\left(\sum\limits_{M<m\le M_1}\bigg|\sum\limits_{L<l\le L_1\atop{\Delta_1<ml\le \Delta_2}}
b(l)e\big((h+\alpha)f_h(l, m)\big)\bigg|^2\right)^{\frac{1}{2}}\nonumber\\
&\ll M^{\frac{1}{2}+\varepsilon}\sum\limits_{|h|\leq H}\big|c_h(\alpha)\big|
\left(\sum\limits_{M<m\le M_1}\frac{L}{Q}\sum_{|q|<Q}\bigg(1-\frac{q}{Q}\bigg)
\sum\limits_{L<l, \, l+q\leq L_1\atop{\Delta_1<ml\le \Delta_2\atop{\Delta_1<m(l+q)\le \Delta_2}}}b(l+q)\overline{b(l)}
e\big(f_h(l, m, q)\big)\right)^{\frac{1}{2}}\nonumber\\
\end{align*}
\begin{align}\label{SIIest1}
&\ll M^{\frac{1}{2}+\varepsilon}\sum\limits_{|h|\leq H}\big|c_h(\alpha)\big|
\Bigg(\frac{L}{Q}\sum\limits_{M<m\le M_1}\Bigg(L^{1+\eta}\nonumber\\
&\hspace{20mm}+\sum_{1\leq |q|<Q}\bigg(1-\frac{q}{Q}\bigg)
\sum\limits_{L<l, \, l+q\leq L_1\atop{\Delta_1<ml\le \Delta_2\atop{\Delta_1<m(l+q)\le \Delta_2}}}
b(l+q)\overline{b(l)}e\big(f_h(l, m, q)\big)\Bigg)^{\frac{1}{2}}\nonumber\\
&\ll X^\varepsilon\sum\limits_{|h|\leq H}\big|c_h(\alpha)\big| \Bigg(\frac{X^2}{Q}+\frac{X}{Q}\sum\limits_{1\leq |q|\leq Q}
\sum\limits_{L<l, \, l+q\leq L_1}\bigg|\sum\limits_{M'<m\leq M_1'}e\big(f_h(l, m, q)\big)\bigg|\Bigg)^{\frac{1}{2}}\,,
\end{align}
where
\begin{equation}\label{M'M1'}
M'=\max{\bigg\{M,\frac{\Delta_1}{l},\frac{\Delta_1}{l+q}\bigg\}}\,,
\quad M_1'=\min{\bigg\{M_1,\frac{\Delta_2}{l},\frac{\Delta_2}{l+q}\bigg\}}
\end{equation}
and
\begin{equation*}
f_h(l, m, q)=(h+\alpha) m^c\Big((l+q)^c \tan^\theta\big(\log (m(l+q))\big)-l^c \tan^\theta\big(\log (ml)\big)\Big)\,.
\end{equation*}
From  \eqref{Conditions2} and \eqref{M'M1'} for the sum  in \eqref{SIIest1} it follows
\begin{equation}\label{M'M1'inM2M}
\left|\begin{array}{ccccc}
L<l, \, l+q\leq L_1\quad\\
M'<m\leq M_1'\quad \quad \;\\
\Delta_1<ml\le \Delta_2 \quad\quad\;\\
\;\Delta_1<m(l+q)\le \Delta_2 \\
(M', M'_1]\subseteq (M, 2M]
\end{array}\right..
\end{equation}
We have
\begin{align}\label{firstderivativem}
\frac{\partial f_h(l, m, q)}{\partial m}
&=(h+\alpha) (q+l)^cm^{c-1}\tan^{\theta-1}\big(\log((q+l)m)\big)\nonumber\\
&\times\Big(c\tan\big(\log((q+l)m)\big)+\theta\sec^2\big(\log((q+l)m)\big)\Big)\nonumber\\
&-(h+\alpha) l^cm^{c-1}\tan^{\theta-1}\big(\log(ml)\big)
\Big(c\tan\big(\log(ml)\big))+{\theta}\sec^2\big(\log(ml)\big)\Big)
\end{align}
and
\begin{align}\label{secondderivativem}
\frac{\partial^2f_h(l, m, q)}{\partial m^2}
&=(h+\alpha) (q+l)^cm^{c-2}\tan^{\theta-2}\big(\log((q+l)m)\big)\nonumber\\
&\times\Big(\big(2\theta\sec^2\big(\log((q+l)m)\big)+c^2-c\big)\tan^2\big(\log((q+l)m)\big)\nonumber\\
&+(2c-1)\theta\sec^2\big(\log((q+l)m)\big)\tan\big(\log((q+l)m)\big)\nonumber\\
&+(\theta^2-\theta))\sec^4\big(\log((q+l)m)\big)\Big)\nonumber\\
&-(h+\alpha)l^cm^{c-2}\tan^{\theta-2}\big(\log(ml)\big)
\Big(\big(2\theta\sec^2\big(\log(ml)\big)+c^2-c\big)\tan^2\big(\log(ml)\big)\nonumber\\
&+(2c-1)\theta\sec^2\big(\log(ml)\big)\tan\big(\log(ml)\big)+(\theta^2-\theta)\sec^4\big(\log(ml)\big)\Big)\,.
\end{align}
From \eqref{Delta1}, \eqref{Delta2}, \eqref{Conditions2}, \eqref{M'M1'inM2M}, \eqref{firstderivativem} and \eqref{secondderivativem} we obtain
\begin{equation}\label{firstderivativemasymp}
\frac{\partial f_h(l, m, q)}{\partial m}\asymp |h+\alpha| L^c M^{c-1}
\end{equation}
and
\begin{equation}\label{secondderivativemasymp}
\frac{\partial^2f_h(l, m, q)}{\partial m^2}\asymp |h+\alpha| L^c M^{c-2}\,.
\end{equation}
Now \eqref{tau}, \eqref{taualphatau2}, \eqref{Conditions2}, \eqref{LMasympX2},
\eqref{SIIest1}, \eqref{M'M1'inM2M}, \eqref{firstderivativemasymp},
\eqref{secondderivativemasymp} and Lemma \ref{GrahamandKolesnik} with  $k=0$  imply
\begin{align*}\label{SIIest2}
S_{II}&\ll X^\varepsilon\sum\limits_{|h|\leq H}\big|c_h(\alpha)\big|\Bigg(\frac{X^2}{Q}+\frac{X}{Q}\sum\limits_{1\leq q\leq Q}
\sum\limits_{L<l\leq L_1}\bigg(|h+\alpha|^{\frac{1}{2}}L^{\frac{c}{2}} M^{\frac{c}{2}}
+|h+\alpha|^{-1}L^{-c}M^{1-c}\bigg)\Bigg)^{\frac{1}{2}}\\
&\ll X^\varepsilon\sum\limits_{|h|\leq H}\big|c_h(\alpha)\big|\Bigg(\frac{X^2}{Q}+X\Big(H^{\frac{1}{2}} L X^{\frac{c}{2}}
+\tau^{-1}L^{1-c}M^{1-c}\Big)\Bigg)^{\frac{1}{2}}\\
&\ll  X^{\frac{11+3c}{15}+\varepsilon}\sum\limits_{|h|\leq H}\big|c_h(\alpha)\big|
\ll X^{\frac{11+3c}{15}+\varepsilon}\sum\limits_{|h|\leq H}\frac{1}{1+|h|}\\
&\ll   X^{\frac{11+3c}{15}+\varepsilon}\,.
\end{align*}
This proves the lemma.
\end{proof}

\begin{lemma}\label{Salphaest} Let 
\begin{equation*}
\tau\leq \alpha \leq 1-\tau\,.
\end{equation*} Then  for the exponential sum denoted by \eqref{Salpha} we have
\begin{equation*}
S(\alpha)\ll  X^{\frac{11+3c}{15}+\varepsilon}\,.
\end{equation*}
\end{lemma}
\begin{proof}
In order to prove the lemma  we will use the formula
\begin{equation}\label{Lambdalog2}
S(\alpha)=S^\ast(\alpha)+\mathcal{O}\big(X^{\frac{1}{2}}\big)\,,
\end{equation}
where
\begin{equation}\label{Sast}
S^\ast(\alpha)=\sum\limits_{\Delta_1<n\leq\Delta_2}\Lambda(n)e\big(\alpha n^c\tan^\theta(\log n)\big)\,.
\end{equation}
By \eqref{Sast} and Lemma \ref{Buriev} with $x=\alpha$, $y=n^c$, $H= X^{\frac{4-3c}{15}}$ we get
\begin{align*}
S^\ast(\alpha)&=\sum\limits_{\Delta_1<n\leq\Delta_2}\Lambda(n)e\big(\alpha n^c\tan^\theta(\log n)-\alpha\{n^c\tan^\theta(\log n)\}\big) \nonumber\\
&=\sum\limits_{\Delta_1<n\leq\Delta_2}\Lambda(n)e\big(\alpha n^c\tan^\theta(\log n)\big)e\big(-\alpha \{n^c\tan^\theta(\log n)\}\big)\nonumber\\
&=\sum\limits_{\Delta_1<n\leq\Delta_2}\Lambda(n)e\big(\alpha n^c\tan^\theta(\log n) \nonumber\\
\end{align*}
\begin{align}\label{Sastformula}
&\times\left(\sum\limits_{|h|\leq H}c_h(\alpha)
e\big(\alpha n^c\tan^\theta(\log n) +\mathcal{O}\Bigg(\min\left(1, \frac{1}{H\|n^c\tan^\theta(\log n)\|}\right)\Bigg)\right)\nonumber\\
&=\sum\limits_{|h|\leq H}c_h(\alpha)\sum\limits_{\Delta_1<n\leq\Delta_2}\Lambda(n)e\big((h+\alpha)n^c\tan^\theta(\log n)\big)\nonumber\\
&+\mathcal{O}\Bigg((\log X)\sum\limits_{\Delta_1<n\leq\Delta_2}\min\left(1, \frac{1}{H\|n^c\tan^\theta(\log n)\|}\right)\Bigg)\nonumber\\
&=S_0^\ast(\alpha)+\mathcal{O}\left((\log X)\sum\limits_{\Delta_1<n\leq\Delta_2}\min\left(1, \frac{1}{H\|n^c\tan^\theta(\log n)\|}\right)\right)\,,
\end{align}
where
\begin{equation}\label{Sast0}
S_0^\ast(\alpha)=\sum\limits_{|h|\leq H}c_h(\alpha)\sum\limits_{\Delta_1<n\leq\Delta_2}\Lambda(n)e\big((h+\alpha)n^c\tan^\theta(\log n)\big)\,.
\end{equation}
Taking into account \eqref{Delta1} and \eqref{Delta2} we have that
\begin{equation}\label{Delta1Delta2inequalities}
\Delta_2<2\Delta_1\,, \quad  2^{-4}X<\Delta_1< 2^2X\,,  \quad  2^{-3}X<\Delta_2< 2^2X\,.
\end{equation} Let
\begin{equation*}
U=2^{-11}X^{\frac{1}{45}}\,,\quad V=2^7X^{\frac{1}{3}}\,,\quad Z=X^{\frac{22}{45}}\,.
\end{equation*}
According to Lemma \ref{Heath-Brown}, the sum $S_0^\ast(\alpha)$
can be decomposed into $O\Big(\log^6\Delta_1\Big)$ sums, each of which is either of Type I
\begin{equation*}
\sum\limits_{|h|\leq H}c_h(\alpha)\mathop{\sum\limits_{M<m\le M_1}a(m)\sum\limits_{L<l\le L_1}}_{\Delta_1<ml\le \Delta_2}
e\big((h+\alpha) m^cl^c \tan^\theta\big(\log (ml) \big)\big)
\end{equation*}
and
\begin{equation*}
\sum\limits_{|h|\leq H}c_h(\alpha)\mathop{\sum\limits_{M<m\le M_1}a(m)\sum\limits_{L<l\le L_1}}_{\Delta_1<ml\le \Delta_2}
e\big((h+\alpha)m^cl^c \tan^\theta\big(\log (ml) \big)\big)\log l\,,
\end{equation*}
where
\begin{equation*}
L\ge Z\,,\quad M_1\le 2M\,,\quad L_1\le 2L\,,\quad a(m)\ll \tau _5(m)\log \Delta_1
\end{equation*}
or of Type II
\begin{equation*}
\sum\limits_{|h|\leq H}c_h(\alpha)\mathop{\sum\limits_{M<m\le M_1}a(m)\sum\limits_{L<l\le L_1}}_{\Delta_1<ml\le \Delta_2}b(l)
e\big((h+\alpha) m^cl^c \tan^\theta\big(\log (ml) \big)\big)\,,
\end{equation*}
where
\begin{equation*}
U\le L\le V\,,\quad M_1\le 2M\,,\quad L_1\le 2L\,,\quad
a(m)\ll \tau _5(m)\log \Delta_1\,,\quad b(l)\ll \tau _5(l)\log \Delta_1\,.
\end{equation*}
Using \eqref{Sast0}, Lemma \ref{SIest} and  Lemma \ref{SIIest} we obtain
\begin{equation}\label{Sast0est}
S_0^\ast(\alpha)\ll  X^{\frac{11+3c}{15}+\varepsilon}\,.
\end{equation}                                                                                                                                   We have
\begin{equation}\label{firstderivativet}
\frac{\partial k y^c\tan^\theta(\log y)}{\partial y} =k y^{c-1}\tan^{\theta-1}(\log y)
\Big(c\tan(\log y)+\theta\sec^2(\log y)\Big)
\end{equation}
and
\begin{align}\label{secondderivativet}
\frac{\partial^2k y^c\tan^\theta(\log y)}{\partial y^2} &=k y^{c-2}\tan^{\theta-2}(\log y)
\Big(\big(2\theta\sec^2(\log y)+c^2-c\big)\tan^2(\log y)\nonumber\\
&+(2c-1)\theta\sec^2(\log y)\tan(\log y)+(\theta^2-\theta)\sec^4(\log y)\Big)\,.
\end{align}
From \eqref{Delta1}, \eqref{Delta2}, \eqref{Delta1Delta2inequalities}, \eqref{firstderivativet} and  \eqref{secondderivativet} it follows
\begin{equation}\label{firstderivativetest}
\frac{\partial k y^c\tan^\theta(\log y)}{\partial y}\asymp  |k|\Delta_1^{c-1} \quad \mbox{ for } \quad y\in [\Delta_1, \Delta_2]
\end{equation}
and
\begin{equation}\label{secondderivativetest}
\frac{\partial^2k y^c\tan^\theta(\log y)}{\partial y^2}\asymp  |k|\Delta_1^{c-2} \quad \mbox{ for } \quad y\in [\Delta_1, \Delta_2]\,.
\end{equation}
Using \eqref{Delta1}, \eqref{Delta2}, \eqref{Delta1Delta2inequalities}, \eqref{firstderivativetest}, \eqref{secondderivativetest},
Lemma \ref{GrahamandKolesnik} with  $k=0$  and Lemma \ref{Expansion} we write
\begin{align}\label{suminest1}
&\sum\limits_{\Delta_1<n\leq\Delta_2}\min\left(1, \frac{1}{H\|n^c\tan^\theta(\log n)\|}\right) \nonumber\\
&=\sum\limits_{\Delta_1<n\leq\Delta_2}\sum\limits_{k=-\infty}^{+\infty}a_k e\big(k n^c\tan^\theta(\log n)\big)
\ll \sum\limits_{k=-\infty}^{+\infty}|a_k| \Bigg|\sum\limits_{\Delta_1<n\leq\Delta_2} e\big(k n^c\tan^\theta(\log n)\big)\Bigg| \nonumber\\
&\ll \frac{X\log 2H}{H}+ \sum\limits_{1\leq k\leq H}\frac{1}{k}\Bigg|\sum\limits_{\Delta_1<n\leq\Delta_2} e\big(k n^c\tan^\theta(\log n)\big)\Bigg|
+ \sum\limits_{ k> H}\frac{H}{k^2}\Bigg|\sum\limits_{\Delta_1<n\leq\Delta_2} e\big(k n^c\tan^\theta(\log n)\big)\Bigg| \nonumber\\
&\ll \frac{X\log 2H}{H}+ \sum\limits_{1\leq k\leq H}\frac{1}{k}\Big(k^{\frac{1}{2}}X^{\frac{c}{2}}+k^{-1}X^{1-c}\Big)
+ \sum\limits_{ k> H}\frac{H}{k^2}\Big(k^{\frac{1}{2}}X^{\frac{c}{2}}+k^{-1}X^{1-c}\Big) \nonumber\\
&\ll X^\varepsilon\left(H^{-1}X+H^{\frac{1}{2}}X^{\frac{c}{2}}+ X^{1-c}\right)\ll  X^{\frac{11+3c}{15}+\varepsilon}\,.
\end{align} Summarizing  \eqref{Lambdalog2}, \eqref{Sastformula},  \eqref{Sast0est} and \eqref{suminest1} we establish the statement in the lemma.
\end{proof}

\begin{lemma}\label{Aest} Let $0<\alpha<1$. Set
\begin{equation}\label{Aalpha}
A(\alpha)=\sum\limits_{\Delta_1<n\leq \Delta_2} e\big(\alpha [n^c\tan^\theta(\log n)]\big)\,.     
\end{equation}
Then
\begin{equation*}
A(\alpha)\ll X^{\frac{1+c}{3}+\varepsilon}+\frac{X^{1-c}}{\alpha}\,.                                                                                        \end{equation*}
\end{lemma}
\begin{proof}
By \eqref{Aalpha} and Lemma \ref{Buriev} with $x=\alpha$, $y=n^c$, $H_0= X^{\frac{2-c}{3}}$ we find
\begin{align}\label{Aalphaformula}                                                                                                                                            
A(\alpha)&=\sum\limits_{\Delta_1<n\leq\Delta_2}e\big(\alpha n^c\tan^\theta(\log n)-\alpha\{n^c\tan^\theta(\log n)\}\big) \nonumber\\
&=\sum\limits_{\Delta_1<n\leq\Delta_2}e\big(\alpha n^c\tan^\theta(\log n)\big)e\big(-\alpha \{n^c\tan^\theta(\log n)\}\big)\nonumber\\
&=\sum\limits_{\Delta_1<n\leq\Delta_2}e\big(\alpha n^c\tan^\theta(\log n) \nonumber\\                                                                                                                                                                                                                                                                                               
&\times\left(\sum\limits_{|h|\leq H_0}c_h(\alpha)
e\big(\alpha n^c\tan^\theta(\log n) +\mathcal{O}\Bigg(\min\left(1, \frac{1}{H_0\|n^c\tan^\theta(\log n)\|}\right)\Bigg)\right)\nonumber\\
&=\sum\limits_{|h|\leq H_0}c_h(\alpha)\sum\limits_{\Delta_1<n\leq\Delta_2}e\big((h+\alpha)n^c\tan^\theta(\log n)\big)\nonumber\\
&+\mathcal{O}\Bigg(\sum\limits_{\Delta_1<n\leq\Delta_2}\min\left(1, \frac{1}{H_0\|n^c\tan^\theta(\log n)\|}\right)\Bigg)\nonumber\\
&=A_0(\alpha)+\mathcal{O}\left(\sum\limits_{\Delta_1<n\leq\Delta_2}\min\left(1, \frac{1}{H_0\|n^c\tan^\theta(\log n)\|}\right)\right)\,,
\end{align}
where
\begin{equation}\label{Aalpha0}
A_0(\alpha)=\sum\limits_{|h|\leq H_0}c_h(\alpha)\sum\limits_{\Delta_1<n\leq\Delta_2}e\big((h+\alpha)n^c\tan^\theta(\log n)\big)\,.
\end{equation}
From \eqref{Delta1}, \eqref{Delta2}, \eqref{Delta1Delta2inequalities}, \eqref{firstderivativetest}, \eqref{secondderivativetest}, \eqref{Aalpha0} and Lemma \ref{GrahamandKolesnik} with  $k=0$  we deduce
\begin{align}\label{Aalpha0est}
A_0(\alpha)&= c_0(\alpha)\sum\limits_{\Delta_1<n\leq\Delta_2}e\big(\alpha n^c\tan^\theta(\log n)\big)
+\sum\limits_{1\leq|h|\leq H_0}c_h(\alpha)\sum\limits_{\Delta_1<n\leq\Delta_2}e\big((h+\alpha)n^c\tan^\theta(\log n)\big) \nonumber\\
&\ll X^{\frac{c}{2}}+\alpha^{-1} X^{1-c} +\sum\limits_{1\leq|h|\leq H_0}\frac{1}{h}\Big((h+\alpha)^{\frac{1}{2}}
X^{\frac{c}{2}}+(h+\alpha)^{-1}X^{1-c}\Big)\nonumber\\
&\ll X^{\frac{c}{2}}+\alpha^{-1} X^{1-c} +H_0^{\frac{1}{2}}X^{\frac{c}{2}}+ X^{1-c}\ll X^{\frac{1+c}{3}}+\alpha^{-1} X^{1-c}\,.
\end{align}
On the other hand  \eqref{suminest1} gives us
\begin{equation}\label{suminest2}
\sum\limits_{\Delta_1<n\leq\Delta_2}\min\left(1, \frac{1}{H_0\|n^c\tan^\theta(\log n)\|}\right)
\ll X^\varepsilon\left(H_0^{-1}X+H_0^{\frac{1}{2}}X^{\frac{c}{2}}\right)\ll  X^{\frac{1+c}{3}+\varepsilon}\,.
\end{equation}
Now the  lemma follows from   \eqref{Aalphaformula}, \eqref{Aalpha0est} and \eqref{suminest2}  \end{proof}
We are now in a good position to estimate $\Gamma_2$ defined by  \eqref{Gamma2}.
We use Cai's \cite{Cai} argument.
From \eqref{Delta1}, \eqref{Delta2} and \eqref{Gamma2} we write
\begin{align}\label{Gamma2est1}
|\Gamma_2|&=\Bigg|\sum\limits_{\Delta_1<p\leq \Delta_2}(\log p)\int\limits_{\tau}^{1-\tau}
S^2(\alpha)e\big(\alpha [p^c\tan^\theta(\log p)]-N\alpha\big)\,d\alpha\Bigg|\nonumber\\
&\leq\sum\limits_{\Delta_1<p\leq \Delta_2}(\log p)\Bigg|\int\limits_{\tau}^{1-\tau}
S^2(\alpha)e\big(\alpha [p^c\tan^\theta(\log p)]-N\alpha\big)\,d\alpha\Bigg|\nonumber\\
&\leq(\log X)\sum\limits_{\Delta_1<n\leq \Delta_2}\Bigg|\int\limits_{\tau}^{1-\tau}
S^2(\alpha)e\big(\alpha [n^c\tan^\theta(\log n)]-N\alpha\big)\,d\alpha\Bigg|\,.
\end{align}
By \eqref{Delta1}, \eqref{Delta2}, \eqref{Gamma2est1} and Cauchy's inequality we deduce
\begin{align}\label{Gamma2est2}
|\Gamma_2|^2
&\leq X(\log X)^2\sum\limits_{\Delta_1<n\leq \Delta_2}\Bigg|\int\limits_{\tau}^{1-\tau}
S^2(\alpha)e\big(\alpha [n^c\tan^\theta(\log n)]-N\alpha\big)\,d\alpha\Bigg|^2\nonumber\\
=& X(\log X)^2\sum\limits_{\Delta_1<n\leq \Delta_2}\int\limits_{\tau}^{1-\tau}S^2(\alpha)
e\big(\alpha [n^c\tan^\theta(\log n)]-N\alpha\big)\,d\alpha \nonumber\\
&\times\int\limits_{\tau}^{1-\tau}\overline{S^2(\beta)e\big(\beta [n^c\tan^\theta(\log n)]-N\beta\big)}\,d\beta  \nonumber\\
&=X(\log X)^2\int\limits_{\tau}^{1-\tau}\overline{S^2(\beta)e(-N\beta)}\,d\beta
\int\limits_{\tau}^{1-\tau}   S^2(\alpha)A(\alpha-\beta)e(-N\alpha)\,d\alpha\nonumber\\
&\leq X(\log X)^2\int\limits_{\tau}^{1-\tau}   |S(\beta)|^2\,d\beta
\int\limits_{\tau}^{1-\tau} |S(\alpha)|^2|A(\alpha-\beta)|\,d\alpha\,.
\end{align}
Using  Lemma \ref{intStautau}, Lemma \ref{Salphaest} and Lemma \ref{Aest} we obtain
\begin{align*}
&\int\limits_{\tau}^{1-\tau} |S(\alpha)|^2|A(\alpha-\beta)|\,d\alpha\nonumber\\
&\ll\int\limits_{\tau\leq\alpha\leq1-\tau\atop{|\alpha-\beta|\leq X^{-c}}}|S(\alpha)|^2|A(\alpha-\beta)|\,d\alpha
+\int\limits_{\tau\leq\alpha\leq1-\tau\atop{|\alpha-\beta|> X^{-c}}}|S(\alpha)|^2|A(\alpha-\beta)|\,d\alpha  \nonumber\\
&\ll  X\int\limits_{\tau\leq\alpha\leq1-\tau\atop{|\alpha-\beta|\leq X^{-c}}}|S(\alpha)|^2\,d\alpha
+\int\limits_{\tau\leq\alpha\leq1-\tau\atop{|\alpha-\beta|> X^{-c}}}|S(\alpha)|^2
\Bigg(X^{\frac{1+c}{3}+\varepsilon}+\frac{X^{1-c}}{|\alpha-\beta|}\Bigg)\,d\alpha\nonumber\\
\end{align*}

\begin{align}\label{IntSalphaAt}
&\ll  X\max\limits_{\tau\leq\alpha\leq 1-\tau}|S(\alpha)|^2\int\limits_{|\alpha-\beta|\leq X^{-c}}\,d\alpha
+X^{\frac{1+c}{3}+\varepsilon}\int\limits_{\tau}^{1-\tau} |S(\alpha)|^2\,d\alpha\nonumber\\
&+ X^{1-c}\max\limits_{\tau\leq\alpha\leq 1-\tau}|S(\alpha)|^2\int\limits_{X^{-c}<|\alpha-\beta|\leq2-2\tau}\frac{1}{|\alpha-\beta|}\,d\alpha\nonumber\\
&\ll  X^{1-c+\varepsilon}\max\limits_{\tau\leq\alpha\leq 1-\tau}|S(\alpha)|^2
+X^{\frac{4+c}{3}+\varepsilon}\nonumber\\
&\ll  X^{\frac{37-9c}{15}+\varepsilon}\,.
\end{align}
Bearing in mind \eqref{Gamma2est2} and \eqref{IntSalphaAt} and  Lemma \ref{intStautau}  we  get \begin{equation}\label{Gamma2est}
\Gamma_2\ll X^{3-c-\varepsilon}\,.
\end{equation}

\subsection{The end of the proof}
Summarizing \eqref{GammaGamma12}, \eqref{Gamma1asymptotic} and \eqref{Gamma2est}
we establish the asymptotic formula \eqref{Asymptoticformula}.

The Theorem is proved.

\vskip18pt
\footnotesize
\begin{flushleft}
S. I. Dimitrov\\
Faculty of Applied Mathematics and Informatics\\
Technical University of Sofia \\
8, St.Kliment Ohridski Blvd. \\
1756 Sofia, BULGARIA\\
e-mail: sdimitrov@tu-sofia.bg\\
\end{flushleft}


\begin{thebibliography}{100}

\bibitem{Baker} R. Baker, {\it Some diophantine equations and inequalities with primes},
Funct. Approx. Comment. Math., \textbf{64} (2), (2021), 203 -- 250.

\bibitem{Buriev} K. Buriev, {\it Additive problems with prime numbers},
Ph.D. Thesis, Moscow State University, (1989), (in Russian).

\bibitem{Cai} Y. Cai, {\it On a Diophantine equation involving primes},
Ramanujan J., {\bf50}, (2019), 151 -- 162.

\bibitem{Dimitrov} S. I. Dimirov, {\it A tangent inequality over primes},
arXiv:2110.08539  [math.NT].

\bibitem{Graham-Kolesnik} S. W. Graham, G. Kolesnik, {\it Van der Corput's Method of Exponential Sums},
Cambridge University Press, New York, (1991).

\bibitem{Heath1} D. R. Heath-Brown, {\it Prime numbers in short intervals and a generalized Vaughan identity},
Canad. J. Math., \textbf{34}, (1982), 1365 -- 1377.

\bibitem{Heath2}  D. R. Heath-Brown, {\it The Piatetski-Shapiro prime number theorem},
J. Number Theory, \textbf{16}, (1983), 242 -- 266.

\bibitem{Iwaniec-Kowalski} H. Iwaniec, E. Kowalski, {\it Analytic number theory},
Colloquium Publications, \textbf{53}, Amer. Math. Soc., (2004).

\bibitem{Kumchev-Nedeva} A. Kumchev, T. Nedeva,
{\it On an equation with prime numbers}, Acta Arith., {\bf 83},
(1998), 117 -- 126.

\bibitem{Laporta-Tolev} M. Laporta, D. Tolev, {\it On an equation with prime numbers},
Mat. Zametki, {\bf57}, (1995), 926 -- 929.

\bibitem{Rassias} M. Th. Rassias,  {\it Goldbach's Problem: Selected Topics}, Springer (2017).

\bibitem{Titchmarsh} E. Titchmarsh, {\it The Theory of the Riemann Zeta-function},
2nd edn., (revised by D. R. Heath-Brown), Clarendon Press, Oxford (1986).

\bibitem{Vaughan}  R. C. Vaughan, {\it The Hardy--Littlewood method},
Sec. ed., Cambridge Univ. Press, (1997).

\bibitem{Vinogradov1} I. M. Vinogradov, {\it Representation of an odd number as the sum of three primes},
Dokl. Akad. Nauk. SSSR, {\bf15}, (1937), 291 -- 294.

\bibitem{Zhai-Cao} W. Zhai, X. Cao, {\it A Diophantine equation with prime numbers},
Acta Math. Sinica, Chinese Series, {\bf45}, (2002), 443 -- 454.

\bibitem{Zhang-Li3} M. Zhang, J. Li, {\it On a Diophantine equation with three prime variables},
Integers, \textbf{19}, \#A 39, (2019).


\end{thebibliography}
\end{document}